\documentclass[a4paper,11pt]{article}

\usepackage{amsmath,amsthm}
\usepackage{amssymb}
\usepackage{breqn}
\usepackage{enumerate}
\usepackage{graphicx}
\usepackage{bm}
\usepackage{bbm}
\usepackage[affil-it]{authblk}
\usepackage{tabu}
\usepackage{bold-extra}
\usepackage[hang,flushmargin]{footmisc}
\usepackage[driverfallback=dvipdfm]{hyperref}
\usepackage{mathtools}
\usepackage{relsize}
\usepackage{scalerel}
\usepackage{xcolor}
\usepackage{titlesec}
\usepackage{apptools}
\usepackage{appendix}

\newtheorem{theorem}{Theorem}

\newtheorem{lemma}[theorem]{Lemma}

\newtheorem{claim}[theorem]{Claim}

\theoremstyle{definition}

\AtAppendix{\counterwithin{theorem}{section}}

\titleformat{\section}[hang]{\scshape\large\bfseries\filcenter}{\S\thesection}{4pt}{}
\titleformat{\subsection}[hang]{\scshape\bfseries}{\thesubsection.}{4pt}{}

\allowdisplaybreaks

\newcommand{\on}[1]{
	\operatorname{#1}
}

\newcommand{\tdt}{\times\cdots\times}

\newcommand{\tightoverset}[2]{
  \mathop{#2}\limits^{\vbox to -.5ex{\kern-1.15ex\hbox{$#1$}\vss}}}

\newcommand{\bigconv}[1]{
	\mathbf{C}_{#1}
}

\newcommand\restr[2]{{
  \left.\kern-\nulldelimiterspace 
  #1 
  \vphantom{\big|} 
  \right|_{#2} 
}}

\makeatletter
\newcommand{\subalign}[1]{%
  \vcenter{%
    \Let@ \restore@math@cr \default@tag
    \baselineskip\fontdimen10 \scriptfont\tw@
    \advance\baselineskip\fontdimen12 \scriptfont\tw@
    \lineskip\thr@@\fontdimen8 \scriptfont\thr@@
    \lineskiplimit\lineskip
    \ialign{\hfil$\m@th\scriptstyle##$&$\m@th\scriptstyle{}##$\hfil\crcr
      #1\crcr
    }%
  }%
}
\makeatother

\newcommand\blfootnote[1]{%
  \begingroup
  \renewcommand\thefootnote{}\footnote{#1}%
  \addtocounter{footnote}{-1}%
  \endgroup
}

\newcommand\ssk[1]{
	\substack{#1}
}

\newcommand{\exx}{
  \mathop{
    \mathchoice{\vcenter{\hbox{\larger[4]$\mathbb{E}$}}}
               {\kern0pt\mathbb{E}}
               {\kern0pt\mathbb{E}}
               {\kern0pt\mathbb{E}}
  }\displaylimits
}

\makeatletter
\newcommand*\bcdot{\mathpalette\bigcdot@{0.5}}
\newcommand*\bigcdot@[2]{\mathbin{\vcenter{\hbox{\scalebox{#2}{$\m@th#1\bullet$}}}}}
\makeatother

\makeatletter
\def\blfootnote{\gdef\@thefnmark{}\@footnotetext}
\makeatother

\newcommand\id{\mathbbm{1}}

\begin{document}
\begin{center}\Large\noindent{\bfseries{\scshape Low-codimensional Subvarieties Inside Dense Multilinear Varieties}}\\[24pt]\normalsize\noindent{\scshape Luka Mili\'cevi\'c\dag}\\[6pt]
\end{center}
\blfootnote{\noindent\dag\ Mathematical Institute of the Serbian Academy of Sciences and Arts\\\phantom{\dag\ }Email: luka.milicevic@turing.mi.sanu.ac.rs}

\footnotesize
\begin{changemargin}{1in}{1in}
\centerline{\sc{\textbf{Abstract}}}
\phantom{a}\hspace{12pt}~Let $G_1, \dots, G_k$ be finite-dimensional vector spaces over a prime field $\mathbb{F}_p$. Let $V$ be a variety inside $G_1 \tdt G_k$ defined by a multilinear map. We show that if $|V| \geq c |G_1| \cdots |G_k|$, then $V$ contains a subvariety defined by at most $K(\log_{p} c^{-1} + 1)$ multilinear forms, where $K$ depends on $k$ only. This result is optimal up to multiplicative constant and is relevant to the partition vs. analytic rank problem in additive combinatorics.
\end{changemargin}
\normalsize

\section{Introduction}

Throughout the paper, we work with a prime field $\mathbb{F}_p$, which is fixed, and finite-dimensional vector spaces $G_1, \dots, G_k$ and $H$ over $\mathbb{F}_p$. We say that a function $\Phi : G_1\times G_2 \tdt G_k \to H$ is a\textit{ multilinear map} if it is linear in each of its $k$ arguments separately. Furthermore, if the codomain is $\mathbb{F}_p$ instead of $H$, we say that $\Phi$ is a \textit{multilinear form}.\\ 

An important goal in additive combinatorics is to obtain a precise understanding of multilinear forms that are not quasirandom. One way of formalizing quasirandomness in this context is to consider the distribution of values of a multilinear form $\alpha : G_1\tdt G_k \to \mathbb{F}_p$. If we fix $x_1, \dots, x_{k-1}$, we may consider the corresponding linear form $y_k \mapsto \alpha(x_1, \dots, x_{k-1}, y_k)$. Since it is a linear form, it is either 0, or it takes all values $0,1, \dots, p-1 \in \mathbb{F}_p$ an equal number of times. Hence, the deviation of distribution of values of $\alpha$ from the uniform distribution corresponds to the frequency of 0 as the value, and a particularly elegant way to express this is to consider the \textit{bias}, defined by
\[\on{bias} \alpha = \exx_{x_1 \in G_1, \dots, x_k \in G_k} \omega^{\alpha(x_1, \dots, x_k)},\]
where $\exx_{x \in X}$ stands for the average over elements $x$ of a set $X$ and $\omega = \exp(2 \pi i / p)$. Hence, small bias corresponds to quasirandom behaviour. A closely related quantity is the \textit{analytic rank}, defined in~\cite{GowWolf} as $\on{arank} \alpha = \log_p \on{bias} \alpha^{-1}$.\\

Let us remark that there are algebraic obstructions to uniform distribution of values. For example, if $\alpha(x_1, \dots, x_k)$ factorizes as $\beta(x_I) \gamma(x_{[k] \setminus I})$, where $\beta$ is a multilinear form on variables with indices in a set $I\subseteq [k]$ and $\gamma$ is a multilinear form on the remaining variables, then it is easy to see that $\on{bias} \alpha \geq 1/p$. A remarkable fact is that this is essentially the only way to have large bias. To make this formal, define the \textit{partition rank}~\cite{Naslund} of a multilinear form $\alpha$, denoted $\on{prank} \alpha$, as the smallest number $r$ such that $\alpha$ can be written as a sum of $r$ multilinear forms that factorize in the way above. It is not hard to see that $\on{prank} \alpha = r$ implies $\on{bias} \alpha \geq p^{-r}$. In the opposite direction we have the following theorem.

\begin{theorem}\label{prankthm}
    For each $k,r \in \mathbb{N}$ there exists a positive integer $K = K(k, r)$ for which the following holds. Let $\alpha : G_1\tdt G_k \to \mathbb{F}_p$ be a multilinear form such that $\on{arank} \alpha \leq r$. Then $\on{prank}\alpha \leq K$.
\end{theorem}

This theorem was first proved in the context of polynomials rather than multilinear forms by Green and Tao~\cite{GreenTaoPolys}. Their approach was refined by Kaufman and Lovett~\cite{KaufmanLovett} and adjusted to the multilinear setting by Bhowmick and Lovett~\cite{BhowLov}. These results gave Ackermann-type dependence on the analytic rank. An improvement was obtained by Janzer~\cite{Janzer1}, which gave tower-type bounds. The dependence was improved to polynomial bounds by Janzer~\cite{Janzer2} and by the author~\cite{LukaRank}. Finally, the current bounds
\[\on{prank} \alpha \leq K(k) \Big(\on{arank} \alpha (\log_p (\on{arank} \alpha + 1) + 1) \Big)\]
were obtained by Moshkovitz and Zhu~\cite{MoshZhu}. See also the discussion in~\cite{MoshZhu} for other related results.\\

The key open problem in this context, stated in~\cite{AdipKazhZie, KazhZie, LamZie, LovettArank}, is to show that the partition rank is linear in the analytic rank. In this paper, we make a contribution to this problem by showing linear bounds for a weaker structural result, which was proved with weaker bounds and played a key role in an earlier proof~\cite{LukaRank} of (polynomial-bounds version of) Theorem~\ref{prankthm}. We will discuss the relevance to the proof in~\cite{LukaRank} after the statement. First, we need some further definitions.\\
\indent By a \textit{multilinear variety} in $G_1 \tdt G_k$ we think of a zero set of a collection of multilinear maps, which are allowed to depend on a proper subset of coordinates. For example, $U \times V \subseteq G_1 \times G_2$ for some subspaces $U \leq G_1$ and $V \leq G_2$ is also a multilinear variety. We say that a multilinear variety $V \subseteq G_1 \tdt G_k$ has \textit{codimension at most} $r$ if it can be represented as a common zero set of at most $r$ multilinear forms.\\ 
\indent The next theorem is our main result. It shows that dense multilinear varieties contain low-codimensional subvarieties. We call it a weak structural result for biased multilinear forms, for the reasons explained after the statement.

\begin{theorem}[Weak structural result with linear bounds] \label{weakinverse} For each $k \in \mathbb{N}$ there exists a positive integer $K = K(k)$ (independent of $p$) for which the following holds. Let $V \subseteq G_{[k]}$ be a multilinear variety of density $c$. Then $V$ contains a multilinear variety of codimension at most $K (\log_{p} c^{-1} + 1)$.\end{theorem}

Going back to Theorem~\ref{prankthm}, the proof in~\cite{LukaRank} was a complicated inductive scheme, where Theorems~\ref{prankthm} and~\ref{weakinverse} were two of several steps that implied each other in the inductive step. In particular, Theorem~\ref{prankthm} for lower arity was used to infer Theorem~\ref{weakinverse} (with polynomial bounds). In this paper, we circumvent this, which opens up the possibility of improving the bounds given by the proof in~\cite{LukaRank}.\\
\indent We remark that biased multilinear forms correspond to dense multilinear varieties. Namely, if $\alpha(x_{[k]})$ has bias $c$, then we may use a dot product $\cdot$ on $G_k$ to write $\alpha(x_{[k]}) = A(x_{[k-1]}) \cdot x_k$ for some multilinear map $A : G_1 \tdt G_{k-1} \to G_k$. The equality $|\{x_{[k-1]} \in G_{[k-1]} : A(x_{[k-1]}) = 0\}| = \on{bias} \alpha |G_{[k-1]}|$ then holds. Thus, for a biased multilinear form $\alpha$, the theorem above allows us to find a multilinear variety of lower order and of essentially optimal codimension inside the zero set of $\alpha$, which is a structural result on $\alpha$.\\
\indent Let us mention a related result of Chen and Ye~\cite{ChenYe}, who recently showed that the geometric rank, which is the algebro-geometric codimension of the variety $\{A = 0\}$ above, is linear in $\on{arank}\alpha$. Our result shows that algebro-geometric codimension of some irreducible component of $\{A = 0\}$ satisfies linear bounds in $\on{arank} \alpha$. On the other hand, our definition of codimension is more informative than its algebro-geometric counterpart. Moreover, our proof is direct and completely combinatorial, while the proof in~\cite{ChenYe} depends on various tools in algebraic geometry, including Lang-Weil bounds.\\

\noindent\textbf{Acknowledgements.} This research was supported by the Ministry of Science, Technological Development and Innovation of the Republic of Serbia through the Mathematical Institute of the Serbian Academy of Sciences and Arts, and by the Science Fund of the Republic of Serbia, Grant No.\ 11143, \textit{Approximate Algebraic Structures of Higher Order: Theory, Quantitative Aspects and Applications} - A-PLUS.

\section{Preliminaries} 

Given an index set $I \subseteq [k]$, we write $G_I = \prod_{i \in I} G_i$. Given a set $X \subseteq G_1 \tdt G_k$, and a tuple $x_I$, consisting of elements $x_i \in G_i$ for $i \in I$, we define the \textit{slice} $X_{x_I}$ to be the set $\{y_{[k] \setminus I} \in G_{[k] \setminus I}: (x_I, y_{[k] \setminus I}) \in X\}$.\\

The next lemma allows us to approximate dense varieties by low-codimensional ones.

\begin{lemma}[Approximating dense varieties externally, Lemma 12 in~\cite{LukaRank}]\label{externalApprox} Let $\Phi \colon G_{[k]} \to H$ be a multilinear map. Then, there is a multilinear map $\phi \colon G_{[k]} \to \mathbb{F}_p^s$ such that $\{\Phi = 0\} \subset \{\phi = 0\}$ and $|\{\phi = 0\} \setminus \{\Phi = 0\}| \leq p^{-s}|G_{[k]}|$.\end{lemma}

The next lemma shows that we may use directional convolutions to fill in very dense subsets of low-codimensional varieties. To state it, we define \textit{convolution in direction} $i$ as the operator that acts on functions $f: G_{[k]} \to \mathbb{R}$ as $\bigconv{i} f(x_{[k]}) = \exx_{y_i \in G_i} f(x_{[i-1]}, y_i + x_i, x_{[i+1, k]}) f(x_{[i-1]}, y_i, x_{[i+1, k]})$.

\begin{lemma}[Directional convolutions of varieties]\label{convvar} Let $W \subseteq G_{[k]}$ be a multilinear variety of codimension $r$. Suppose that $B \subseteq W$ is a subset of size $|B| \leq 2^{-2k} p^{-kr} |G_{[k]}|$. Then
\[\bigconv{k}\bigconv{k-1} \dots \bigconv{1} \id_{W \setminus B}(x_{[k]}) > 0\]
holds for all $x_{[k]} \in W$.  
\end{lemma}

\begin{proof}
    We prove the claim by induction on $k$. The base case $k = 1$ is a standard fact in additive combinatorics. Turning to the inductive step, assume the claim holds for $k - 1$. Fix arbitrary $x_{[k]} \in W$. Let $A \subseteq G_k$ be the set of all $y_k$ such that $|B_{y_k}| \geq 2^{-2(k-1)}p^{-(k-1)r} |G_{[k-1]}|$. By averaging, $|A| \leq \frac{1}{4p^r}|G_k|$. Let $y_k \in W_{x_{[k-1]}} \setminus A$ be arbitrary. Hence, since $y_k \notin A$ and $W_{y_k} \not= \emptyset$ is a multilinear variety in $G_{[k-1]}$ of codimension at most $r$, then we may apply inductive hypothesis to obtain 
    \[\bigconv{k-1} \dots \bigconv{1} \id_{W_{y_k} \setminus B_{y_k}}(x_{[k-1]}) > 0.\]
    Since $|A| < \frac{1}{2} |W_{x_{[k-1]}}|$, there exists $y_k$ such that $y_k, y_k + x_k$ both belong to $W_{x_{[k-1]}} \setminus A$. Thus
    \[\bigconv{k}\bigconv{k-1} \dots \bigconv{1} \id_{W \setminus B}(x_{[k]}) \geq 0.\qedhere\]
\end{proof}

\section{Proof of Theorem~\ref{weakinverse}}

We prove the theorem by induction on the arity $k$ of the multilinear maps in question. The following lemma is of key importance for the induction step. Assuming the inductive hypothesis, it allows us to find a bounded codimension multilinear variety $W$ such that fibers of the given variety $V$ associated to points of $W$ are dense. We use $O(\cdot)$ notation to hide implicit constants that depend on $k$ only, and not on $p$.

\begin{lemma}\label{denseColumnsLemma}Suppose that Theorem~\ref{weakinverse} for arity $k$ holds with constant $K = K(k)$. Let $V \subset G_{[k+1]}$ be a variety of density $c$. Then there exists a multilinear variety $W \subseteq G_{[k]}$ of codimension at most $K (\log_{p} c^{-1} + 1)$ such that for each $x_{[k]} \in W$ we have $|V_{x_{[k]}}| \geq (c/p)^{O(1)}  |G_{k+1}|$.\end{lemma}

\begin{proof}
 Let $c' > 0$ be a parameter to be specified later. Let us pick a random $x_{k+1} \in G_{k+1}$ and consider the corresponding slice $U = V_{x_{k+1}}$ in $G_{[k]}$. Let $B \subseteq U$ be the set of all $x_{[k]} \in U$ such that $|V_{x_{[k]}}| \leq c' |G|$. Note that these points have sparse fibres of $V$, thus we think of them as bad points. By linearity of expectation, we have
\begin{align*}\exx |U| \,\,-\,\, \frac{c{c'}^{-1}}{2} |B| = &\sum_{x_{[k]} \in G_{[k]}} \mathbb{P}\Big(x_{[k]} \in V_{x_{k+1}}\Big) - \frac{c{c'}^{-1}}{2} \sum_{\ssk{x_{[k]} \in G_{[k]}\\|V_{x_{[k]}}| \leq c' |G_{k+1}|}} \mathbb{P}\Big(x_{[k]} \in V_{x_{k+1}}\Big)\\
=&\sum_{x_{[k]} \in G_{[k]}} \mathbb{P}\Big(x_{k+1} \in V_{x_{[k]}}\Big) - \frac{c{c'}^{-1}}{2} \sum_{\ssk{x_{[k]} \in G_{[k]}\\|V_{x_{[k]}}| \leq c' |G_{k+1}|}} \mathbb{P}\Big(x_{k+1} \in V_{x_{[k]}}\Big)\\
=&\sum_{x_{[k]} \in G_{[k]}} \frac{|V_{x_{[k]}}|}{|G_{k + 1}|} - \frac{c{c'}^{-1}}{2} \sum_{\ssk{x_{[k]} \in G_{[k]}\\|V_{x_{[k]}}| \leq c' |G_{k+1}|}} \frac{|V_{x_{[k]}}|}{|G_{k+1}|} \\
\geq & \,|G_{k+1}|^{-1} |V|\, - \,\frac{c}{2} |G_{[k]}| \,\,\geq \,\,\frac{c}{2} |G_{[k]}|.\end{align*}

Hence, there is a choice of element $x_{k+1} \in G_{k+1}$ for which $|U| \geq \frac{c}{2} |G_{[k]}|$ and $|B| \leq 2c^{-1} c' |G_{[k]}|$. Applying induction hypothesis, i.e. Theorem~\ref{weakinverse} for arity $k$, to $U$, we find a multilinear variety $W \subseteq U$ of codimension $r \leq K (\log_{p} c^{-1} + 2)$.\\ 

We set $c' =2^{-2k - 1} p^{-2kK} c^{k K + 1}$. This ensures that $c' \leq 2^{-k} c p^{-kr}$, allowing us to use Lemma~\ref{convvar} to obtain $\bigconv{k}\bigconv{k-1} \dots \bigconv{1} \id_{W \setminus B}(x_{[k]}) > 0$ for all $x_{[k]} \in W$. It follows that whenever $x_{[k]} \in W$, then $V_{x_{[k]}} \geq {c'}^{2^k} |G_{k+1}|$, since we have intersection of $2^k$ subspaces $V_{a_{[k]}}$ of density at least $c'$, where $a_{[k]}$ ranges over $2^k$ points of a parallelepiped in $W \setminus B$ that attests to $\bigconv{k}\bigconv{k-1} \dots \bigconv{1} \id_{W \setminus B}(x_{[k]}) > 0$.\end{proof}

We are now ready to prove Theorem~\ref{weakinverse}.

\begin{proof}[Proof of Theorem~\ref{weakinverse}] As mentioned before, we prove the theorem by induction on $k$. The base case is $k = 1$ when $V$ becomes a subspace of density $c$. Then $V$ is a subspace of codimension exactly $\log_{p} c^{-1}$, completing the proof of the base of induction, with $K(1) = 1$.\\

We now turn to the inductive step. Suppose that Theorem~\ref{weakinverse} holds for arity $k$ with constant $K = K(k)$. For each $i \in [k+1]$ apply Lemma~\ref{denseColumnsLemma} to get a multilinear variety $W^{(i)} \subseteq G_{[k+1] \setminus \{i\}}$ of codimension $r \leq K (\log_{p} c^{-1} + 2)$ such that $|V_{x_{[k+1] \setminus \{i\}}}| \geq c' |G_i|$ for all $x_{[k+1] \setminus \{i\}} \in W^{(i)}$, where $c' \geq (c/p)^{O(1)}$. We slightly modify varieties $W^{(i)}$ by defining further multilinear varieties $U^{(i)}$ as zero sets of all multilinear forms that do not depend on variable $x_i$ that appear in definition of some $W^{(j)}$. Thus, we allow repeated use of multilinear forms defining varieties $W^{(j)}$ and the codimension of $U^{(i)}$ is at most $(k+1)r$. Set $\tilde{V} = V \cap \Big(\cap_{i \in [k+1]} G_i \times U^{(i)}\Big)$. Note that variety $\tilde{V}$ has the property that whenever $x_{[k + 1] \setminus \{i\}} \in U^{(i)}$ then $|\tilde{V}_{x_{[k+1] \setminus \{i\}}}| \geq c'' |G_i|$, where $c'' \geq (c/p)^{O(1)}$. This follows from the fact that $\tilde{V}_{x_{[k+1] \setminus \{i\}}}$ is intersection of the subspace $V_{x_{[k+1] \setminus \{i\}}}$, which is dense as $x_{[k+1] \setminus \{i\}} \in U^{(i)} \subseteq W^{(i)}$, with at most $k$ subspaces, namely $U^{(j)}_{x_{[k+1] \setminus \{i\}}}$, of codimension $(k+1)r$. Let $\varepsilon > 0$ be a parameter to be specified later. Use Lemma~\ref{externalApprox} to approximate $\tilde{V}$ externally up to error density $\varepsilon$ by a multilinear variety $A$ of codimension $\log_{p} \varepsilon^{-1}$. Set $\tilde{A} =A \cap \Big(\cap_{i \in [k+1]} G_i \times U^{(i)}\Big)$. Note that $|\tilde{A} \setminus \tilde{V}| \leq |A \setminus \tilde{V}| \leq \varepsilon |G_{[k+1]}|$. We claim that $\tilde{A} = \tilde{V}$.

\begin{claim}If $\varepsilon < {c''}^{\,k+1}$ then varieties $\tilde{A}$ and $\tilde{V}$ are the same.\end{claim}

\begin{proof}Suppose that $\tilde{V} \subsetneq \tilde{A}$. There exists a point $x_{[k+1]} \in \tilde{A} \setminus \tilde{V}$. By induction on $i \in [0, k+1]$ we show that there exist at least ${c''}^{\,i} |G_{[i]}|$ of choices of $y_{[i]}$ such that $(y_{[i]}, x_{[i+1, k+1]}) \in \tilde{A} \setminus \tilde{V}$. The base case $i = 0$ is trivial. Supposing the claim holds for some $i$, take any $(y_{[i]}, x_{[i+1, k+1]}) \in \tilde{A} \setminus \tilde{V}$. Hence, $\tilde{V}_{y_{[i]}, x_{[i+2, k+1]}}$ is a proper subspace of $\tilde{A}_{y_{[i]}, x_{[i+2, k+1]}}$. Therefore, the number of $y_{i+1}$ such that $(y_{[i+1]}, x_{[i+2, k+1]}) \in \tilde{A} \setminus \tilde{V}$ is at least $|\tilde{V}_{y_{[i]}, x_{[i+2, k+1]}}|$, which is at least $c''|G_{i+1}|$, since $(y_{[i]}, x_{[i+1, k+1]}) \in \tilde{A}$, so $(y_{[i]}, x_{[i+2, k+1]}) \in U^{(i+1)}$. Having proved the inductive step, we obtain $|\tilde{A} \setminus \tilde{V}| \geq {c''}^{\,k+1}|G_{[k+1]}|$, which is a contradiction with the choice of $A$.\end{proof}

Take $\varepsilon =  {c''}^{\,k + 1} /2$. Since $\tilde{A}$ has codimension $\log_p \varepsilon^{-1} + (k+1)^2 r \leq O(\log_{p} c^{-1} + 1)$, we are done.\end{proof}

\end{document}